\documentclass[12pt]{article}
\usepackage{amsmath,amssymb}
\usepackage{amsthm}

\def\PC{\mathcal{P}}
\def\MC{\mathcal{M}}

\def\LC{\mathcal{L}}

\def\R{\mathbf{R}}

\def\1{\mathbf{1}}

\def\tr{\rm{tr}}

\def\al{\alpha}
\def\be{\beta}
\def\pa{\partial}
\def\ep{\epsilon}
\def\de{\delta}

\newtheorem{prop}{Proposition}[section]
\newtheorem{theorem}{Theorem}[section]

\newtheorem{remark}{Remark}

\newcommand{\si}{\sigma}

\newcommand{\om}{\omega}

\begin{document}
\title{Nonlinear L\'evy and nonlinear Feller processes: an analytic introduction\thanks{
Supported by the AFOSR grant FA9550-09-1-0664 'Nonlinear Markov control processes and games'}}
\author{Vassili N. Kolokoltsov\thanks{Department of Statistics, University of Warwick,
 Coventry CV4 7AL UK,
  Email: v.kolokoltsov@warwick.ac.uk}}
\maketitle

\begin{abstract}
The program of studying general nonlinear Markov processes was put forward in
\cite{Ko07}. This program was developed by the author in monograph \cite{Ko10nonlbook},
where, in particular, nonlinear L\'evy processes were introduced.
 The present paper is an invitation to the rapidly developing topic of this monograph.
We provide a quick (and at the same time more abstract) introduction to the basic analytical aspects
of the theory developed in Part II of \cite{Ko10nonlbook}.
\end{abstract}

\section{Introduction}

Nonlinear L\'evy processes were introduced by the author in \cite{Ko10nonlbook}.
We provide a quick introduction to the basic analytical aspects
of the theory developed in Part II of \cite{Ko10nonlbook} giving more concise and more general
formulations of some basic facts on well-posedness and sensitivity of nonlinear processes.
For general background in L\'evy and Markov processes we refer to books
\cite{Ap}, \cite{Ko11Markbook}, \cite{Kyp}.

For sensitivity of the nonlinear jump-type processes, e.g. Boltzmannn or Smoluchovski, we refer
to papers \cite{Ko06a} and \cite{Bail}.

Loosely speaking, a nonlinear Markov evolution is just a dynamical system generated by
a measure-valued ordinary differential equation (ODE) with the
specific feature of preserving positivity. This feature
distinguishes it from a general Banach space valued ODE
and yields a natural link with probability theory, both in
interpreting results and in the tools of analysis.
Technical complications for the sensitivity analysis,
again compared with the standard theory of vector-valued ODE, lie in the specific unboundedness of generators
that causes the derivatives of the solutions to nonlinear equations (with respect to parameters or initial conditions)
to live in other spaces, than the evolution itself. From the probabilistic point of view, the
first derivative with respect to initial data (specified by the
 linearized evolution around a path of nonlinear dynamics)
 describes the interacting particle approximation to this nonlinear dynamics
 (which, in turn, serves as the dynamic law of large numbers to this approximating
 Markov system of interacting particles), and the second derivative describes the limit of fluctuations
 of the evolution of particle systems around its law of large numbers
 (probabilistically the dynamic central limit theorem). In this paper we concentrate only on the analytic aspects
 of the theory referring to \cite{Ko10nonlbook} for probabilistic interpretation.

Recall first the definition of a propagator.
For a set $S$, a family of mappings $U^{t,r}$, from $S$ to itself,
parametrized by the pairs of real numbers $r \leq t$ (resp. $t \leq r$)
from a given finite or infinite interval is called a {\it forward
propagator}\index{propagator} (resp. a {\it backward propagator}\index{backward propagator}),
 if $U^{t,t}$ is the identity operator in $S$ for all $t$ and
the following {\it chain rule}\index{chain rule}, or {\it propagator equation}\index{propagator equation},
holds for $r \leq s \leq t$ (resp. for $t \leq s \leq r$): $U^{t,s}
U^{s,r} = U^{t,r}$. If the
mappings $U^{t,r}$ forming a backward propagator depend only on the
differences $r-t$, then the family $T^t=U^{0,t}$ forms a semigroup.
That is why, propagators are sometimes referred to as two-parameter semigroups.
By a propagator we mean a forward or a backward propagator (which should be clear from the context).

 Let $\tilde \MC(X)$ be a dense
subset of the space $\MC(X)$ of finite (positive Borel) measures on
a polish (complete separable metric) space $X$ (considered in its weak topology). By a nonlinear
{\it sub-Markov}\index{propagator!sub-Markov} (resp. {\it Markov})
{\it propagator}\index{propagator!Markov} in $\tilde \MC(X)$ we
shall mean any propagator $V^{t,r}$ of possibly nonlinear
transformations of $\tilde \MC(X)$ that do not increase (resp.
preserve) the norm. If $V^{t,r}$ depends only on the difference $t-r$
and hence specifies a semigroup, this semigroup is called nonlinear or
generalized {\it sub-Markov} or {\it Markov}
respectively.

The usual, linear, Markov propagators or semigroups correspond to the
case when all the transformations are linear contractions in the
whole space $\MC(X)$. In probability theory these propagators
describe the evolution of averages of Markov processes, i.e.
processes whose evolution after any given time $t$ depends on the
past $X_{\le t}$ only via the present position $X_t$. Loosely speaking, to any nonlinear Markov propagator there
corresponds a process whose behavior after any time $t$ depends on
the past $X_{\le t}$ via the position $X_t$ of the process and its
distribution at $t$.

More precisely, consider the nonlinear
kinetic equation\index{kinetic equation}
\begin{equation}
\label{eqkineqmeanfield1}
 {d \over dt} (g, \mu_t)
 =(A[\mu_t]g, \mu_t)
\end{equation}
with a certain family of operators $A[\mu]$ in $C(X)$ depending
on $\mu$ as a parameter and such that each $A[\mu]$ specifies a uniquely defined Markov process
(say, via solution to the corresponding martingale problem, or by generating a
Feller semigroup).

 Suppose that the Cauchy problem for equation \eqref{eqkineqmeanfield1}
is well posed and specifies the weakly continuous Markov semigroup
$T_t$ in $\MC(X)$. Suppose also that for any weakly continuous curve
$\mu_t \in \PC(X)$ (the set of probability measures on $X$) the solutions to the Cauchy problem of the
equation
\begin{equation}
\label{eqnonhomogweakprop}
 \frac{d}{dt} (g,\nu_t)=(A[\mu_t]g, \nu_t)
\end{equation}
define a weakly continuous propagator $V^{t,r}[\mu_.]$, $r\le t$, of
linear transformations in $\MC(X)$ and hence a Markov process in
$X$, with transition probabilities $p^{[\mu_.]}_{r,t}(x,dy)$. Then to any $\mu\in \PC(X)$
there corresponds a (usual linear, but time non-homogeneous) Markov process $X_t^{\nu}$ in $X$ ($\nu$ stands for an initial distribution) such that its distributions $\nu_t$ solve equation \eqref{eqnonhomogweakprop} with the initial condition $\nu$. In particular,
the distributions of $X_t^{\mu}$ (with the initial condition $\mu$) are $\mu_t=T_t(\mu)$ for all
times $t$ and the transition probabilities
$p^{[\mu_.]}_{r,t}(x,dy)$ specified by equation
\eqref{eqnonhomogweakprop} satisfy the condition
\begin{equation}
\label{eqnonhomogweaktransnonlinearprop}
 \int_{X^2} f(y)p^{[\mu_.]}_{r,t}(x,dy)\mu_r(dx)=(f,V^{t,r}\mu_r)=(f,\mu_t).
\end{equation}
 We shall call the family of processes $X_t^{\mu}$ a {\it nonlinear
Markov process}\index{nonlinear Markov process}.
When each $A[\mu]$ generates a
Feller semigroup and $T_t$ acts on the whole $\MC(X)$ (and not only on its dense subspace),
the corresponding process can be also called {\it nonlinear Feller}.
Allowing for the evolution on subsets $\tilde \MC(X)$ is however crucial, as it often occurs in applications,
say for the Smoluchovski or Boltzmann equation with unbounded rates.

Thus a nonlinear Markov process is a semigroup of the transformations of distributions
such that to each trajectory is attached a ``tangent'' Markov process with the same marginal distributions.
The structure of these tangent processes is not intrinsic to the semigroup, but can be specified by
choosing a stochastic representation for the generator, that is of the r.h.s. of \eqref{eqnonhomogweakprop}.

In this paper we shall prove a general well-posedness result for nonlinear Markov semigroups that will cover, as particular cases,

(i) {\it nonlinear L\'evy processes} specified by the families
\[
 A_{\mu}f(x)=\frac{1}{2}(G(\mu) \nabla,
\nabla)f(x)+(b(\mu), \nabla f)(x)
\]
\begin{equation}
\label{eqLevygennonlinearrep} +\int [f(x+y)-f(x)-(y, \nabla f
(x))\1_{B_1}(y)]\nu(\mu,dy),
\end{equation}
where, for each probability measure $\mu$ on $\R^d$, $\nu (\mu,.)$ is a L\'evy measure (i.e. a Borel measure on $\R^d$ without a mass point at the origin and such that the function $\min (1,|y|^2)$ is integrable with respect to it), $G(\mu)$ is a symmetric non-negative $d\times d$-matrix, $b(\mu)$ a vector in $\R^d$ and $B_1$ is the unit ball in $\R^d$ with $\1_{B_1}$ being the corresponding indicator function;

(ii) processes {\it of order at most one} specified by the families
\begin{equation}
\label{eqshortgeneratornonlin}
 A_{\mu}f(x)=(b(x,\mu),\nabla f(x))+\int_{\R^d}
 (f(x+y)-f(x))\nu (x,\mu,dy),
 \end{equation}
 where the L\'evy measure $\nu$ is supposed to have a finite first moment;

 (iii)  mixtures of possibly degenerate diffusions and stable-like
processes and processes generated
by the operators of order at most one, explicitly defined below in Proposition
\ref{propstablelikeprocdegrep}.

It is worth noting that equations of type \eqref{eqnonhomogweakprop} that
appear naturally as dynamic Law of Large Numbers for interacting particles, can be deduced,
on the other hand, from the mere assumption of positivity
preservation, see \cite{Ko10nonlbook} and \cite{Str03}. In case of diffusion
(partial second order) operators $A[\mu]$, the corresponding evolution
\eqref{eqkineqmeanfield1} was first analyzed by McKean and is often called the {\it McKean}
or {\it McKean-Vlasov} diffusion. Its particular case that arises
 as the limit of grazing collisions in the Boltzmann collision model is sometimes referred to as the {\it Landau-Fokker-Planck}
 equation, see \cite{GMN} for some recent results.
 The case of $A[\mu]$ being a Hamiltonian vector field is often called a {\it Vlasov-type} equation,
 as it contains the celebrated {\it Vlasov equation} from plasma physics.
 The case of $A[\mu]$ being pure integral operators comprises a large variety of models from
 statistical mechanics (say, Boltzmann and Smoluchovskiu equations) to evolutionary games (replicator dynamics),
 see \cite{Ko10nonlbook} for a comprehensive review and papers
\cite{Frank}, \cite{Zak00}, \cite{Zak02} for the introduction to nonlinear Markov evolutions
from the physical point of view.

The following basic notations will be used:

 $C_\infty(\R^d)\subset C(\R^d)$ consists of $f$ such that  $\lim_{x\to
 \infty}f(x)=0$,

 $C^k(\R^d)$ (resp. $C^k_{\infty}(\R^d)$) is the Banach space of $k$
times continuously differentiable functions with bounded derivatives
on $\R^d$ (resp. its closed subspace of functions $f$ with $f^{(l)} \in C_\infty(\R^d)$, $l \le k$) with
 \[
  \|f\|_{C^k(\R^d)}= \sum_{l=0}^k \|f^{(l)}\|_{C(\R^d)},
  \]

  $\PC(\R^d)$ the set of probability measures on $\R^d$.

$\|A\|_{D\to B}$ denotes the norm of an operator $A$ in the Banach space $\LC(D,B)$
of bounded linear operators between Banach spaces $D$ and $B$, and $\|\xi\|_B$ denotes the norm of $\xi$
 as an element of the Banach space $B$.

\section{Dual propagators}

  A backward propagator
$\{U^{t,r}\}$ of uniformly (for $t,r$ from a compact set) bounded linear
operators on a Banach space $B$ is called {\it strongly continuous}\index{strongly continuous propagator} if
the family $U^{t,r}$ depends strongly continuously on $t$ and $r$.

For a strongly continuous backward propagator $\{U^{t,r}\}$ of bounded linear
operators on a Banach space $B$ with a common invariant domain $D\subset B$,
 which is itself a Banach space
with the norm $\| \, \|_D \ge \| \, \|_B$,
let $\{A_t\}$, $t \ge 0$, be a family of bounded linear operators $D\to B$
depending strongly measurably on $t$ (i.e. the function $t\mapsto A_t f \in B$ is measurable for each $f\in D$).
Let us say that the family $\{A_t\}$ {\it generates}
$\{U^{t,r}\}$ on the invariant domain $D$\index{propagator!generator family} if the equations
\begin{equation}
\label{eqdefnonhomgenerator}
  \frac{d}{ds} U^{t,s} f = U^{t,s} A_s f,
 \quad \frac{d}{ds} U^{s,r}f = - A_s U^{s,r}f, \quad t \le s \le r,
\end{equation}
 hold a.s. in $s$ for any $f\in D$, that is
there exists a negligible subset $S$ of $\R$ such that for all $t<r$ and all $f\in D$
equations \eqref{eqdefnonhomgenerator} hold for all $s$ outside $S$,
where the derivatives exist in the Banach topology of
$B$. In particular, if the operators $A_t$ depend strongly continuously on $t$ (as bounded operators $D\to B$),
this implies that equations
\eqref{eqdefnonhomgenerator} hold for all $s$ and $f\in D$,
where for $s=t$ (resp. $s=r$) it is assumed to be only a
right (resp. left) derivative.

 For a Banach space $B$ or a linear operator $A$ one usually denotes by
$B^{\star}$ or $A^{\star}$ its {\it Banach dual} (space or operator respectively). Alternatively
the notations $B'$ and $A'$ are in use.

\begin{theorem}({\bf Basic duality})\index{duality}
\label{thdualpropag}

 Let $\{U^{t,r}\}$ be a strongly continuous
backward propagator of bounded linear operators in a Banach space
$B$ with a common invariant domain $D$, which is itself a Banach space
with the norm $\| \, \|_D \ge \| \, \|_B$, and let the family $\{A_t\}$ of bounded linear operators $D\to B$
generate $\{U^{t,r}\}$ on $D$. Then

(i) the family of dual
operators $V^{s,t}=(U^{t,s})^{\star}$ forms a weakly-${\star}$ continuous in
$s,t$ propagator of bounded linear operators in $B^{\star}$
(contractions if all $U^{t,r}$ are contractions) such that
\begin{equation}
\label{eqdefnonhomgeneratordual}
  \frac{d}{dt} V^{s,t} \xi = -V^{s,t}A_t^{\star} \xi,
 \quad
 \frac{d}{ds} V^{s,t} \xi = A_s^{\star}V^{s,t} \xi, \quad t \le s,
\end{equation}
hold weakly-$\star$ in $D^{\star}$, i.e., say, the second equation means
\begin{equation}
\label{eqdefnonhomgeneratordual1}
 \frac{d}{ds}(f, V^{s,t} \xi)
  =(A_sf, V^{s,t} \xi), \quad t \le s, \quad f\in D;
\end{equation}

(ii) $V^{s,t} \xi$ is the unique solution to the Cauchy problem of
equation \eqref{eqdefnonhomgeneratordual1} in $B^{\star}$, i.e. if $\xi_t=\xi$ for
a given $\xi \in B^{\star}$ and $\xi_s$, $s\in [t,r]$, is a weakly-${\star}$
continuous family in $B^{\star}$ satisfying
\begin{equation}
\label{eqdefnonhomgeneratordual2}
 \frac{d}{ds}(f, \xi_s)
  =(A_sf, \xi_s), \quad t \le s \le r, \quad f\in D,
\end{equation}
then $\xi_s=V^{s,t}\xi$ for $t\le s \le r$.

(iii) $U^{s,r}f$ is the unique solution to the inverse Cauchy
problem of the second equation in \eqref{eqdefnonhomgenerator}.
\end{theorem}

\begin{proof} Statement (i) is a direct consequence of duality.

(ii) Let
$g(s)=(U^{s,r}f,\xi_s)$ for a given $f\in D$. Writing
\[
(U^{s+\delta,r}f,\xi_{s+\delta})-(U^{s,r}f,\xi_s)
\]
\[
=(U^{s+\delta,r}f-U^{s,r}f,\xi_s)
 +(U^{s,r}f,\xi_{s+\delta}-\xi_s)
 \]
 \[
 +(U^{s+\delta,r}f-U^{s,r}f,\xi_{s+\delta}-\xi_s)
 \]
 and using \eqref{eqdefnonhomgenerator}, \eqref{eqdefnonhomgeneratordual1}
 and the invariance of $D$,
 allows one to conclude that
 \[
 \frac{d}{ds}g(s)=-(A_sU^{s,r}f,\xi_s)+(U^{s,r}f,A^{\star}_s\xi_s)=0,
 \]
 because a.s. in $s$
 \[
 \left(\frac{U^{s+\delta,r}f-U^{s,r}f}{\delta},\xi_{s+\delta}-\xi_s\right) \to 0,
 \]
 as $\de \to 0$ (since the family $\de^{-1}(U^{s+\delta,r}f-U^{s,r}f)$ is relatively compact, being convergent,
 and $\xi_s$ is weakly continuous).
 Hence $g(r)=(f,\xi_r)=g(t)=(U^{t,r}f,\xi_t)$,
  showing that $\xi_r$ is uniquely defined.

(iii) is proved similar to (ii).
\end{proof}

\begin{remark}
\label{remarktothdualpropag}
In addition to the statement of Theorem \ref{thdualpropag} let us note
(as one sees directly from duality), that
(i) $V^{s,t} \xi$ depend weakly-$\star$ continuous on $s,t$ uniformly for bounded $\xi$
and (ii) $V^{s,t}$ is a weakly-$\star$ continuous operator, that is $\xi_n \to \xi$ weakly-$\star$
implies $V^{s,t}\xi_n \to V^{s,t}\xi$ weakly-$\star$.
\end{remark}

\begin{remark}
\label{remarktothdualpropag}
Working with discontinuous $A_t$ is crucial for the development of the related theory
of SDE with nonlinear noise, see \cite{Ko08a} and  \cite{Ko10}. In this paper we shall use only
continuous families of generators $\{A_t\}$.
\end{remark}

We deduce now some corollaries of Theorem \ref{thdualpropag}: on the extension of the operators $V^{s,t}$ to $D^{\star}$,
and on their stability with respect to a perturbation of the family $A_t$.

 \begin{theorem}
\label{thdualpropagext}
Under the assumptions of Theorem \ref{thdualpropag} suppose additionally that

(i) $\{U^{t,s}\}$ is a strongly continuous backward propagator of uniformly bounded operators in $D$;

(ii) there exists another subspace $\tilde D\subset D$, dense in $D$, which is itself a Banach space
with the norm $\| \, \|_{\tilde D} \ge \| \, \|_D$ such that the mapping $t\mapsto A_t$ is a continuous mapping
$t\to \LC(\tilde D,D)$;

(iii) $B^{\star}$ is dense in $D^{\star}$ (which holds automatically in case of reflexive $D$).

Then the operators $V^{s,t}: B^{\star} \to B^{\star}$ extend to the operators $V^{s,t}: D^{\star} \to D^{\star}$
forming a weakly-$\star$ continuous propagator in $D^{\star}$ that solves equation \eqref{eqdefnonhomgeneratordual1}
weakly-$\star$ in $\tilde D^{\star}$, that is, for any $\xi \in D^{\star}$, equation \eqref{eqdefnonhomgeneratordual1} holds for all $f\in \tilde D$.
\end{theorem}

\begin{proof}
The fact that $V^{s,t}$ extend to linear operators in $D^{\star}$ follows without any additional
assumption from the invariance of $D$ under $U^{t,s}$. Assumption (i) implies that this extension
is bounded and weakly-$\star$ continuous in $D^{\star}$. In order to prove that \eqref{eqdefnonhomgeneratordual1} holds for $f\in \tilde D$ and $\xi \in D^{\star}$, observe that
\begin{equation}
\label{eqdefnonhomgeneratordual1int}
(f, V^{r,t} \xi)=(f,\xi)+
  \int_t^r(A_sf, V^{s,t} \xi) \, ds
\end{equation}
for $\xi \in B^{\star}$, $f \in D$.
Now, for a $\xi \in D^{\star}$ and $f\in \tilde D$, let us pick up a sequence
 $\xi_n \in B^{\star}$ such that $\xi_n \to \xi$ in the norm topology of $D^{\star}$ as $n\to \infty$ (which is possible by assumption (iii)). As $A_s f\in D$ (by assumption(ii)), we can pass to the limit in
 \eqref{eqdefnonhomgeneratordual1int} with $\xi_n$ instead of $\xi$ (using dominated convergence) yielding
  \eqref{eqdefnonhomgeneratordual1int} for $\xi \in D^{\star}$ and $f\in \tilde D$. Finally,
  as $(A_sf, V^{s,t} \xi)$ is a continuous function of $s$ (by assumption (ii) and the weak-$\star$ continuity of $V^{s,t}$ in $D^{\star}$), equation \eqref{eqdefnonhomgeneratordual1int} implies
    \eqref{eqdefnonhomgeneratordual1} for $\xi \in D^{\star}$ and $f\in \tilde D$.
\end{proof}

 \begin{theorem}
\label{thdualpropagext1}
Under the assumptions of Theorem  \ref{thdualpropagext} assume additionally that
the backward propagator $\{U^{t,s}\}$ in $D$ is generated by $\{A_t\}$ on the invariant domain $\tilde D$
(in particular $\tilde D$ is invariant and equations \eqref{eqdefnonhomgenerator} hold in
the norm topology of $D$ for any $f\in \tilde D$). Then
$V^{s,t}\xi$ represents the unique weakly-$\star$ continuous in $D^{\star}$ solution of equation \eqref{eqdefnonhomgeneratordual1}
in $\tilde D^{\star}$. Moreover, for the propagator $\{U^{t,s}\}$ in $D$ to be generated by $\{A_t\}$ on $\tilde D$
it is sufficient to assume that $\{U^{t,s}\}$ is a strongly continuous family of bounded operators in $\tilde D$.
\end{theorem}

\begin{proof}
The first statement is a direct consequence of Theorem \ref{thdualpropag} applied to the pair
of spaces $\tilde D, D$. The last statement is proved as in the previous theorem.
Namely, we first rewrite equation \eqref{eqdefnonhomgenerator} in the integral form, i.e. as
  \begin{equation}
\label{eqdefnonhomgeneratorint}
U^{t,r} f =f+ \int_t^r A_s U^{s,r}f \, ds,
\quad
 U^{t,r} f =f+ \int_t^r U^{t,s} A_s f \, ds.
\end{equation}
These equations would imply \eqref{eqdefnonhomgenerator} with the derivative defined in the norm topology of $D$,
for $f\in \tilde D$, if we can prove that the functions $A_s U^{s,r}f$ and $U^{t,s} A_s f$ are continuous
functions $s\mapsto D$. To see that this is true, say for the first function, we can write
\[
A_{s+\de} U^{s+\de,r}f-A_s U^{s,r}f
=A_{s+\de} (U^{s+\de,r}f-U^{s,r}f)+ (A_{s+\de}-A_s) U^{s,r}f.
\]
The first term tends to zero in the norm topology of $D$, as $\de \to 0$, by
the strong continuity of $U^{s,r}$ in $\tilde D$, and the second term tends to zero by the continuity
of the family $A_s$ (assumption (ii) of Theorem \ref{thdualpropagext}).
\end{proof}

We conclude this section with a simple result on the convergence of propagators.

\begin{theorem}
\label{thconvergepropagators}
Suppose we are given a sequence of backward propagators
$\{U^{t,r}_n\}$, $n=1,2,...$, generated by the families  $\{A^n_t\}$
and a backward propagator $\{U^{t,r}\}$ generated by the family $\{A_t\}$. Let all these propagators
satisfy the same conditions as $U^{t,r}$ and $A_t$ from Theorem
\ref{thdualpropag} with the same $D$, $B$. Suppose also that all $U^{t,r}$
are uniformly bounded as operators in $D$.

Assume finally that, for any $t$ and any $f\in D$,  $A_t^n f$ converge to $A_t f$, as $n\to \infty$,
in the norm topology of $B$. Then  $U^{t,r}_n$ converges to $U^{t,r}$
strongly in $B$.
Moreover,
\begin{equation}
\label{eq4thconvergepropagators}
\|(V^{r,t}_n-V^{r,t})\xi \|_{D^{\star}} \le c \|A_s^n -A_s\|_{D\to B}\|\xi\|_{B^{\star}}.
\end{equation}
\end{theorem}

\begin{proof} By the density argument (taking into account that $U^{t,r}_ng$ are uniformly bounded in $B$),
in order to prove the strong convergence of $U^{t,r}_n$ to $U^{t,r}$, it is sufficient to prove that
$U^{t,r}_ng$ converges to $U^{t,r}g$ for any $g\in D$. But if $g\in D$,
\begin{equation}
\label{eq3thconvergepropagators}
(U^{t,r}_n-U^{t,r})g= U^{t,s}_nU^{s,r}g\mid_{s=t}^r
=\int_t^rU^{t,s}_n(A_s^n-A_s)U^{s,r}g\,
 ds,
\end{equation}
which converges to zero in the norm topology of $B$ by the dominated convergence.
Estimate \eqref{eq4thconvergepropagators} also follows from \eqref{eq3thconvergepropagators}.
\end{proof}

\section{Perturbation theory for weak propagators}

 The main point of the perturbation theory is to build a propagator
generated by the family of operators $\{A_t+F_t\}$, when a propagator $U^{t,r}$ generated by
$\{A_t\}$ is given and $\{F_t\}$ are bounded. However, if $\{F_t\}$ are only bounded,
 then instead of the solutions to the equation
 \begin{equation}
\label{eqdefnonhomgenper}
 \frac{d}{ds}f
  =A_sf +F_s f, \quad t \le s \le r,
\end{equation}
with a given terminal $f_r$,
as desired, one can only construct the solutions to the so called {\it mild form} of this equation:
\begin{equation}
\label{eqgenmildequationnonhom}
 f_t=U^{t,r}f+\int_t^r U^{t,s}F_sf_s \, ds,
\end{equation}
which is only formally equivalent to \eqref{eqdefnonhomgenper} (i.e. when a solution to the mild equation is regular enough which may not be the case).

Let us recall the simplest perturbation theory result for propagators, which clarifies this issue
(a proof can be found e.g. in \cite{Ko11Markbook}, Theorem 1.9.3, and simpler, but similar fact for semigroups is discussed in almost any text book on functional analysis).

\begin{theorem}
\label{thperturbationtheoryforpropagators}
(i) Let $U^{t,r}$ be a strongly continuous backward propagator of bounded linear operators in
a Banach space $B$, and $\{F_t\}$
be a family of bounded operators in $B$ depending strongly continuous on $t$. Set
\begin{equation}
\label{eqperturbseriesforprop}
 \Phi^{t,r}=U^{t,r}
 +\int_t^r U^{t,s}F_sU^{s,r} \, ds
 +\sum_{m=1}^{\infty} \int_{t\le s_1\le \cdots \le s_m\le r}
 U^{t,s_1}F_{s_1}U^{s_1,s_2}  \cdots  F_{s_m}U^{s_m,r}\, ds_1  \cdots  ds_m.
 \end{equation}
It is claimed that this series converges in $B$ and the family
$\{\Phi^{t,r}\}$ also forms a strongly continuous propagator of bounded operators in $B$
such that $f_t=\Phi^{t,s} f$ is the unique solution to equation \eqref{eqgenmildequationnonhom}.

(ii) Suppose additionally that a family of linear operators $\{A_t\}$
 generates $\{U^{t,r}\}$ on the common invariant
domain $D$, which
is dense in $B$ and is itself a
Banach space under a norm $\|.\|_D\ge \|.\|_B$.
Suppose that $U^{t,r}$ and $\{F_t\}$ are also uniformly bounded operators in $D$. Then $D$ is
 invariant under $\{\Phi^{t,r}\}$ and
the family $\{A_t+F_t\}$ generates $\{\Phi^{t,r}\}$ on $D$.
Moreover, series \eqref{eqperturbseriesforprop} also converges in the operator norms of $D$
 and operators $\Phi^{t,r}f$ are bounded as operators in the Banach space $D$.
\end{theorem}

We presented this theorem, because for the sensitivity analysis of nonlinear equations we shall need non-homogeneous extensions of equations \eqref{eqdefnonhomgeneratordual2} of the form
\begin{equation}
\label{eqdefnonhomgeneratordual3}
 \frac{d}{ds}(f, \xi_s)
  =(A_sf, \xi_s) +(F_s f,\xi_s), \quad t \le s \le r,
\end{equation}
where $F_s$ is a family of operators bounded in $D$, but, what is crucial and necessitates
technical complications, not bounded in $B$.

Under the assumption of Theorem \ref{thdualpropagext1} and assuming $\{F_t\}$ is a bounded strongly continuous
family of operators in $D$, it follows directly from Theorem \ref{thperturbationtheoryforpropagators} (ii)
applied to the pair of Banach spaces $(D, \tilde D)$ that the perturbation theory propagator
\eqref{eqperturbseriesforprop} solves equation \eqref{eqdefnonhomgenper} in $D$ and is generated
on $\tilde D$ by the family $\{A_t+F_t\}$. Hence, by Theorem \ref{thdualpropag}, the dual
propagator $\{\Psi^{r,t}=(\Phi^{t,r})'\}$ is weakly-$\star$ continuous in $D^{\star}$ and yields
a unique solution to  \eqref{eqdefnonhomgeneratordual3} in $\tilde D^{\star}$ (i.e. so that,
for $\xi_s=\Psi^{s,t}\xi_t$, equation \eqref{eqdefnonhomgeneratordual3} holds for all $f\in \tilde D$).

 The next result proves the same fact, except for uniqueness, under weaker assumptions of
 Theorem \ref{thdualpropagext}.

 \begin{theorem}
\label{thperturbweak}
 Under the assumptions of Theorem \ref{thdualpropagext} assume $\{F_t\}$ is a bounded strongly continuous
family of operators in $D$. Let $\{\Phi^{t,r}\}$ be given by
\eqref{eqperturbseriesforprop}, which by Theorem \ref{thperturbationtheoryforpropagators} (i)
(applied to the pair of Banach spaces $(D, \tilde D)$) is a strongly continuous propagator
in $D$, and let $\{\Psi^{r,t}=(\Phi^{t,r})'\}$, which is clearly
a weakly-$\star$ continuous backward propagator in $D^{\star}$.
Then the curve $\xi_s=\Psi^{s,t}\xi_t$ solves equation \eqref{eqdefnonhomgeneratordual3} in $\tilde D^{\star}$
with a given terminal condition $\xi_t$, that is
 \eqref{eqdefnonhomgeneratordual3} holds for all $f\in \tilde D$.
 \end{theorem}

 \begin{proof}
 From duality and \eqref{eqperturbseriesforprop} it follows that
 \begin{equation}
\label{eqperturbseriesforpropdual}
 \Psi^{r,t}=V^{r,t}+\sum_{m=1}^{\infty} \int_{t\le s_1\le \cdots \le s_m\le r}
 V^{r,s_m} F'_{s_m} \cdots V^{s_2,s_1}F'_{s_1} V^{s_1,t}\, ds_1  \cdots  ds_m,
 \end{equation}
 where $F'_s$ are of course dual operators to $F_s$, and
 where the integral is understood in weak-$\star$ sense and the series converges in
 the norm-topology of $D^{\star}$ (we need to take into account Remark \ref{remarktothdualpropag} to see that the weak integral is well defined). To prove \eqref{eqdefnonhomgeneratordual3} for $f\in \tilde D$ we should now differentiate
 term by term the corresponding series $(f,\Psi^{r,t}\xi)$ with respect to $r$ using Theorem \ref{thdualpropagext}. This term-by-term
 differentiation is then justified by the fact that the series of derivatives
 \[
 (A_rf, V^{r,t}\xi_t) +\left[(F_rf, V^{r,t}\xi_t)+\int_t^r (A_rf, V^{r,s} F'_s V^{s,t})\, ds\right]
 +\cdots
 \]
 converges uniformly in $r$.
 \end{proof}

\section{$T$-products}

Here we shall recall the notion of $T$-products showing how they can be used to construct propagators
generated by families of operators each of which generates a sufficiently regular semigroup.

We shall work with three Banach spaces $B_0,B_1,B_2$ with the norms denoted by $\|\,\|_i$, $i=0,1,2$,
such that $B_0\subset B_1\subset B_2$, $B_0$ is dense in $B_1$, $B_1$ is dense in $B_2$ and $\|\,\|_0\ge \|\,\|_1 \ge \|\,\|_2$.

 Let $L_t :B_1 \mapsto B_2$, $t\ge 0$, be a family of
uniformly (in $t$) bounded operators such that the closure in
$B_2$ of each $L_t$ is the generator of a strongly continuous
semigroup of bounded operators in $B_2$. For a partition
$\Delta=\{0=t_0<t_1<...<t_N=t \}$ of an interval $[0,t]$ let us
define a family of operators $U_{\Delta}(\tau,s)$, $0\le s \le
\tau \le t$, by the rules
$$
U_{\Delta}(\tau,s)=\exp \{ (\tau-s)L_{t_j} \},
\quad t_j \le s \le \tau \le t_{j+1},
$$
$$
U_{\Delta}(\tau,r)=U_{\Delta}(\tau,s)U_{\Delta}(s,r), \quad
0\le r \le s \le \tau \le t.
$$
Let $\Delta t_j=t_{j+1}-t_j$ and $\delta (\Delta)=\max_j \Delta t_j$.
If the limit
\begin{equation}
\label{eq1secbasicTproduct}
U(s,r)f=\lim_{\delta (\Delta)\to 0} U_{\Delta} (s,r)f
\end{equation}
exists for some $f$ and all $0\le r \le s \le t$ (in the norm of
$B_2$), it is called the $T$-{\it product}\index{$T$-product}
(or {\it chronological exponent}) of
$L_t$ and is denoted by $T \exp \{\int_r^s L_{\tau} \, d\tau
\}f$. Intuitively, one expects the $T$-product to give a solution
to the Cauchy problem
\begin{equation}
\label{eq2secbasicTproduct}
{d \over dt}\phi =L_t \phi, \quad \phi_0=f,
\end{equation}
in $B_2$ with the initial conditions $f$ from $B_1$.

\begin{theorem}
\label{thbasicTproduct}
 Let a family $L_tf$, $t\ge 0$, of linear operators in $B_2$ be given such that

 (i) each $L_t$ generates a strongly continuous semigroup $e^{sL_t}$, $s\ge 0$, in $B_2$
  with invariant core $B_1$,

 (ii) $L_t$  are uniformly bounded operators $B_0\to B_1$ and $B_1\to B_2$,

 (iii) $B_0$ is also invariant under all $e^{sL_t}$ and these operators are uniformly bounded as operators in $B_0,B_1$, $B_2$, with the norms not exceeding $e^{Ks}$ with a constant $K$ (the same for all $B_j$ and $L_t$),

 (iv) $L_tf$, as a function $t\mapsto B_2$,
 depends continuously on $t$ locally uniformly
in $f$ (i.e. for $f$ from bounded subsets of $B_1$).

 Then

(i) the $T$-product $T \exp \{\int_0^s L_{\tau} \, d\tau \}f$ exists for all $f \in B_2$,
and the convergence in \eqref{eq1secbasicTproduct} is uniform in $f$ on any bounded subset of $B_1$;

(ii) if $f\in B_0$, then the approximations $U_{\Delta}(s,r)$ converge also in $B_1$;

(iii) this $T$-product defines a strongly continuous (in $t,s$)
family of uniformly bounded operators in both $B_1$ and $B_2$,

(iv) this $T$-product $T \exp \{\int_0^s L_{\tau} \, d\tau \}f$ is a solution of
problem \eqref{eq2secbasicTproduct} for any $f \in B_1$.
\end{theorem}

\begin{proof}
(i) Since $B_1$ is dense
in $B_2$ and all $U_{\Delta}(s,r)$ are uniformly bounded in $B_2$ (by (iii)), the existence of the $T$-product for all $f\in B_2$ follows from its existence for $f\in B_1$. In the latter case it follows from the formula
$$
U_{\Delta}(s,r)-U_{\Delta'}(s,r)
=U_{\Delta'}(s,\tau)U_{\Delta}(\tau,r)|_{\tau =r}^{\tau =s}
=\int_r^s {d \over d\tau} U_{\Delta'} (s,\tau) U_{\Delta}(\tau,r)
\, d\tau
$$
$$
=\int_r^s U_{\Delta'}(s,\tau) (L_{[\tau]_{\Delta}}-L_{[\tau]_{\Delta' }})
U_{\Delta}(\tau,r) \, d\tau
$$
(where we denoted $[s]_{\Delta}=t_j$ for $t_j\le s <t_{j+1}$), because $L_t$  are uniformly continuous
 (condition (iv))
and $U_{\Delta}(s,r)$ are uniformly bounded in $B_2$ and $B_1$ (by condition (iii)).

(ii) If $f\in B_0$, then
the equations
\[
U_{\Delta}(s,r)
=\int_r^s L_{[\tau]_{\Delta }}
U_{\Delta}(\tau,r) \, d\tau,
\]
imply that the family $U_{\Delta}(s,t)$ is uniformly Lipschitz continuous in $B_1$ as a function of $t$,
because $L_s$  are uniformly bounded operators $B_0\to B_1$ and $U_{\Delta}(s,r)$ are uniformly bounded in $B_0$.
 Hence one can choose a subsequence, $U_{\Delta_n}(s,r)$, converging
in $C([0,T],B_1)$. But the limit is unique (it is the limit in $B_2$), implying the convergence of the whole family $U_{\Delta}(s,t)$, as $\de (\Delta) \to 0$.

(iii) It follows from (iii) that the limiting propagator is bounded. Strong continuity in $B_1$ is deduced first for $f\in B_0$ and
then for all $f\in B_1$ by the density argument.

(iv) If $f\in B_0$, we
can pass to the limit in the above approximate equations to obtain the equation
\[
U(s,r)f
=\int_r^s L_{\tau}
U(\tau,r) f \, d\tau.
\]
Since $B_0$ is dense in $B_1$, we then deduce the same equation for an arbitrary $f \in B_1$.
This implies that $U(s,r)f$ satisfies equation \eqref{eq2secbasicTproduct} by condition (iv) and the basic theorem of calculus.
\end{proof}

To conclude the section we present a rather general example of a non-homogeneous generator
of a strongly continuous Markov propagator specifying a time nonhomogeneous Feller process.
This will be a time-nonhomogeneous possibly degenerate
diffusion combined with a mixture of possibly degenerate stable-like
processes and processes generated
by the operators of order at most one, that is a process generated by an operator of the form
\[
 L_tf (x) =\frac{1}{2} {\tr} (\sigma_t (x) \sigma_t ^T(x) \nabla^2
 f(x))+(b_t(x), \nabla f(x))+\int (f(x+y)-f(x)) \nu_t (x,dy)
\]
\begin{equation}
\label{eqgendifstablelikedegenper}
 + \int_P (dp) \int_0^K d|y| \int_{S^{d-1}}
 a_{p,t}(x,s)\frac {f(x+y)-f(x)-(y,\nabla f(x))}{|y|^{\al_{p,t} (x,s)+1}}
 \om_{p,t}(ds).
\end{equation}
Here $s=y/|y|$, $K>0$ and $(P,dp)$ is a Borel space with a finite
measure $dp$ and $\om_{p,t}$ are certain finite Borel measures on
$S^{d-1}$.

\begin{prop}
\label{propstablelikeprocdegrep}
 Let the functions $\si, b,a, \al$ and the finite measure $|y| \nu (x,dy)$ be of smoothness class $C^5$ with respect to all variables (the measure is smooth in the weak sense), and
  $a_p,\al_p$ take
values in compact subintervals of $(0,\infty)$ and $(0,2)$
respectively. Then the family of operators $L_t$ of form
\eqref{eqgendifstablelikedegenper} generates a backward propagator $U_{t,s}$ on the invariant domain $C_{\infty}^2(\R^d)$,
and hence a unique
Markov process.
\end{prop}

\begin{proof} For a detailed proof (that uses several ingredients
 including Theorem \ref{thbasicTproduct} as a final step) we refer to the book \cite{Ko11Markbook}.
\end{proof}

 \section{Nonlinear propagators}
\label{secnonlinprop}

The following result from \cite{Ko10nonlbook} represents the basic tool
allowing one to build nonlinear propagators from infinitesimal linear
ones.

Recall that $V^{s,t}$ denotes the dual of $U^{t,s}$ given by Theorem \ref{thdualpropag}.
Let $M$ be a bounded subset of $B^{\star}$ that is closed in the
norm topologies of both $B^{\star}$ and $D^{\star}$.
For a $\mu
\in M$ let $C_{\mu}([0,r],M)$ be the metric space of the continuous
in the norm $D^{\star}$ curves $\xi_s\in M$, $s\in [0,r]$,
$\xi_0=\mu$, with the distance
\[
 \rho (\xi_.,\eta_.)=\sup_{s\in [0,r]}\|\xi_s-\eta_s\|_{D^{\star}}.
 \]

\begin{theorem}
\label{thbasicnonlinearwellpose}
 (i) Let $D$ be a dense subspace of a
Banach space $B$ that is itself a Banach space such that $\|f\|_D\ge
\|f\|_B$, and let $\xi \mapsto A[\xi]$ be a
mapping from $B^{\star}$ to bounded linear operators $A[\xi]: D\to B$
such that
\begin{equation}
\label{eqAisLiponxi}
 \|A[\xi]-A[\eta]\|_{D\to B} \le
 c\|\xi-\eta\|_{D^{\star}}, \quad \xi, \eta \in
 B^{\star}.
\end{equation}

(ii)
For any $\mu \in M$ and $\xi_.\in C_{\mu}([0,r],M)$, let
  the operator curve $A[\xi_t]: D\to B$ generate a strongly
continuous backward propagator of uniformly bounded linear operators
 $U^{t,s}[\xi_.]$ in $B$, $0\le t\le s\le r$, on the common invariant domain $D$
(in particular, \eqref{eqdefnonhomgenerator} holds), such that
\begin{equation}
\label{eqpropagatorboundedonasubspace}
 \|U^{t,s}[\xi_.]\|_{D\to D} \le c, \quad t \le s\le r,
\end{equation}
for some constant $c>0$ and with their dual propagators $V^{s,t}[\xi_.]$ preserving the set $M$.

 Then the weak nonlinear Cauchy problem
\begin{equation}
\label{eqgennonlinearprob}
 \frac{d}{dt}(f,\mu_t)=(A[\mu_t]f,\mu_t), \quad \mu_0=\mu, \quad
 f\in D,
\end{equation}
is well posed in $M$. More precisely, for any $\mu \in M$ it has a
unique solution $T_t(\mu) \in M$, and the transformations $T_t$ of
$M$ form a semigroup for $t\in [0,r]$ depending Lipschitz
continuously on time $t$ and the initial data in the norm of
$D^{\star}$, i.e.
\begin{equation}
\label{eqLipcontnonlinearsemigroup}
 \|T_t(\mu)-T_t(\eta)\|_{D^{\star}} \le c(r,M) \|\mu-\eta \|_{D^{\star}},
 \quad
\|T_t(\mu)-\mu\|_{D^{\star}} \le c(r,M)t
\end{equation}
with a constant $c(r,M)$.
\end{theorem}

 \begin{proof}  Since
\[
(f,(V^{t,0}[\xi^1_.]-V^{t,0}[\xi^2_.])\mu)=
(U^{0,t}[\xi^1_.]f-U^{0,t}[\xi^2_.]f,\mu)
\]
and
\[
U^{0,t}[\xi^1_.]-U^{0,t}[\xi^2_.]=
U^{0,s}[\xi^1_.]U^{s,t}[\xi^2_.]\mid_{s=0}^t \]
\[
=\int_0^tU^{0,s}[\xi^1_.](A[\xi^1_s]-A[\xi^2_s])U^{s,t}[\xi^2_.]\,
 ds,
\]
and taking into account \eqref{eqAisLiponxi} and
\eqref{eqpropagatorboundedonasubspace} one deduces that
\[
\|(V^{t,0}[\xi^1_.]-V^{t,0}[\xi^2_.])\mu \|_{D^{\star}}
 \le \|U^{0,t}[\xi^1_.]-U^{0,t}[\xi^2_.]\|_{D\to B}\|\mu\|_{B^{\star}}
 \]
 \[
 \le tc(r,M)\sup_{s\in [0,r]}
 \|\xi^1_s-\xi^2_s\|_{D^{\star}}
 \]
 (of course we used the assumed boundedness of $M$),
 implying that for $t\le t_0$ with a small enough $t_0$ the mapping $\xi_t \mapsto
 V^{t,0}[\xi_.]$ is a contraction in $C_{\mu}([0,t],M)$.
 Hence by the contraction principle there exists a unique fixed point
 for this mapping. To obtain the unique global solution one just has
 to iterate the construction on the next interval
 $[t_0,2 t_0]$, then on $[2t_0,3t_0]$, etc. The semigroup property of
 $T_t$ follows directly from uniqueness.

Finally, if $T_t(\mu)=\mu_t$ and $T_t(\eta)=\eta_t$, then
\[
\mu_t-\eta_t=V^{t,0}[\mu_.]\mu-V^{t,0}[\eta_.]\eta
=(V^{t,0}[\mu_.]-V^{t,0}[\eta_.])\mu + V^{t,0}[\eta_.](\mu-\eta).
\]
Estimating the first term as above yields
\[
\sup_{s\le t}\|\mu_s-\eta_s\|_{D^{\star}}
 \le c(r, M) (t \sup_{s\le t}\|\mu_s-\eta_s\|_{D^{\star}}
 +\|\mu-\eta\|_{D^{\star}}),
\]
which implies the first estimate in
\eqref{eqLipcontnonlinearsemigroup} first for small times, which is
then extended to all finite times by the iteration. The second
estimate in \eqref{eqLipcontnonlinearsemigroup} follows
 from \eqref{eqdefnonhomgeneratordual1}.
\end{proof}

\begin{remark}
\label{remarkbasicnonlinearwellpose}
 For our purposes, the basic
examples are given by  $B=C_{\infty}(\R^d)$, $M=\PC(\R^d)$, and
$D=C^2_{\infty}(\R^d)$ or $D=C^1_{\infty}(\R^d)$.
In order to see that $\PC(\R^d)$ is closed in the norm topology of $D^{\star}$ for $D=C^k_{\infty}(\R^d)$
with any natural $k$, observe that
 the distance $d$ on $\PC(\R^d)$ induced by
its embedding in $(C^k_{\infty}(\R^d))'$ is defined by
\[
d(\mu,\eta)=\sup \{|(f,\mu-\eta)|: f\in C^2_{\infty}(\R^d),
\|f\|_{C^2_{\infty}(\R^d)} \le 1\}.
\]
and hence
\[
d(\mu,\eta)=\sup \{|(f,\mu-\eta)|: f\in C^2(\R^d), \|f\|_{C^2(\R^d)}
\le 1\}.
\]
Consequently, convergence $\mu_n \to \mu$, $\mu_n \in \PC(\R^d)$,  with respect to
this metric implies the convergence $(f,\mu_n)\to (f,\mu)$ for all
$f\in C^k(\R^d)$ and hence for all $f\in C_{\infty}(\R^d)$ and for
$f$ being constants. This implies tightness of the family $\mu_n$
and that the limit $\mu \in \PC(\R^d)$.
\end{remark}

Theorem \ref{thbasicTproduct} supplies a useful criterion for condition (ii) of the previous theorem,
thus yielding the following corollary.

\begin{theorem}
\label{thbasicnonlinearwellpose1}

Under the assumption (i) of Theorem \ref{thbasicnonlinearwellpose} assume instead of (ii) the following:

(ii') There exists another
 Banach space $\tilde D$, which is a dense subspace of $D$, so that
 all $A[\mu]$, $\mu \in M$, are uniformly bounded operators $\tilde D\to D$ and $D\to B$.

 (iii') For any $\mu \in M$
  the operator $A[\mu]: D\to B$ generates a strongly
 continuous semigroup $e^{tA[\mu]}$ in $B$
  with invariant core $D$, such that $\tilde D$ is also invariant under all $e^{sA[\mu]}$, and these operators are uniformly bounded as operators in $\tilde D,D$, $B$, with the norms not exceeding $e^{Ks}$ with a constant $K$,

  (iv') the set $M$ is invariant under all dual semigroups $(e^{tA[\mu]})'$.

Then condition (ii) and hence the conclusion of Theorem \ref{thbasicnonlinearwellpose} hold.
Moreover, the operators $U^{t,s}[\mu_.]$ form a strongly continuous propagator of
bounded operators in $D$.
\end{theorem}

\begin{proof}
For $\xi_.\in C_{\mu}([0,r],M)$,
the operator curve $L_s=A[\xi_s]: D\to B$ clearly satisfies conditions (i)-(iii)
of Theorem \ref{thbasicTproduct}. To check its last condition (iv) we have to show
that $A[\xi_t]f$ as a function $t\mapsto B$ is continuous uniformly for $f$ from a bounded domain of $D$.
And this follows from \eqref{eqAisLiponxi}, as it implies
\[
 \|(A[\xi_t]-A[\xi_s])f \|_B \le
 c\|\xi_t-\xi_s\|_{D^{\star}} \|f\|_D.
\]
Hence Theorem \ref{thbasicTproduct} is applicable to the curve $L_s=A[\xi_s]: D\to B$,
implying condition (ii) of Theorem \ref{thbasicnonlinearwellpose}.
\end{proof}

As a preliminary step in studying sensitivity, let us prove a simple
stability result for the above nonlinear
semigroups $T_t$ with respect to the small perturbations of the
generator.

\begin{theorem}
\label{thbasicnonlinearwellposestabil} Under the assumptions of
Theorem \ref{thbasicnonlinearwellpose} suppose $\xi\mapsto \tilde
A[\xi]$ is another mapping from $B^{\star}$ to bounded operators
$D\to B$ satisfying the same condition as $A$ with the corresponding
$\tilde U^{t,s}$, $\tilde V^{s,t}$ satisfying the same
conditions as $U^{t,s}$, $V^{s,t}$. Suppose
\begin{equation}
\label{eqthbasicnonlinearwellposestabil1}
 \|\tilde A[\xi]-A[\xi]\|_{D\to B} \le
 \kappa, \quad \xi \in M
\end{equation}
with a constant $\kappa$. Then
\begin{equation}
\label{eqthbasicnonlinearwellposestabil2}
 \|\tilde T_t(\eta)-T_t(\mu)\|_{D^{\star}} \le
 c(r, M)
(\kappa +\|\mu-\eta\|_{D^{\star}}).
\end{equation}
\end{theorem}

\begin{proof} As in the proof of Theorem \ref{thbasicnonlinearwellpose},
denoting $T_t(\mu)=\mu_t$ and $\tilde T_t(\eta)=\tilde \eta_t$ one
can write
\[
\mu_t-\tilde \eta_t=(V^{t,0}[\mu_.])- \tilde V^{t,0}[\tilde
\eta_.])\mu +\tilde V^{t,0}[\tilde \eta](\mu-\eta)
\]
and then
\[
\sup_{s\le t}\|\mu_s-\tilde \eta_s\|_{D^{\star}} \le c(r,M)\left( t
(\sup_{s\le
t}\|\mu_s-\tilde\eta_s\|_{D^{\star}}+\kappa)+\|\mu-\eta\|_{D^{\star}}\right),
\]
which implies \eqref{eqthbasicnonlinearwellposestabil2} first for
small times, and then for all finite times by iterations.
\end{proof}

\section{Linearized evolution around a path of a nonlinear semigroup}

Both for numerical simulations and for the application to interacting particles, it is crucial
to analyze the dependence of the solutions to nonlinear kinetic equations on some parameters
 and on the initial data. Ideally we would like to have smooth dependence.

More precisely, suppose we are given a family of operators $A^{\al}[\mu]$, depending on a real parameter $\al$ and
 satisfying the assumptions of Theorem
 \ref{thbasicnonlinearwellpose} for each $\al$. For
 $\mu_t^{\al}=\mu_t^{\al}(\mu_0^{\al})$, a solution to corresponding \eqref{eqkineqmeanfield1} with the initial condition $\mu_0^{\al}$,
  we are interested in the derivative
\begin{equation}
\label{eqdefderwithpar}
 \xi_t(\al)=\frac{\pa \mu_t^{\al}}{\pa \al}.
\end{equation}

In this section we shall start with the analysis of the
linearized evolution around a path of a nonlinear semigroup.
Namely, differentiating \eqref{eqkineqmeanfield1} (at least formally for the moment)
with respect to $\al$ yields the equation
\begin{equation}
\label{eqforderpar}
 \frac{d}{dt}(g,\xi_t(\al))
 =(A^{\al}[\mu_t^{\al}]g, \xi_t (\al))+(D_{\xi_t(\al)}A^{\al}[\mu_t^{\al}]g,\mu_t^{\al})
  + \left(\frac{\pa A^{\al}[\mu_t^{\al}]}{\pa \al}g,\mu_t^{\al}\right),
\end{equation}
with the initial condition
\begin{equation}
\label{eqdefderwithparinit}
\xi_0=\xi_0(\al)=\frac{\pa \mu_0^{\al}}{\pa \al},
\end{equation}
where
\begin{equation}
\label{eqdeffirstGatderwithinitial}
 D_{\eta}A^{\al}[\mu]
 = \lim_{s \rightarrow 0_+} \frac{1}{s}(A^{\al}[\mu+s\eta] - A^{\al}[\mu])
\end{equation}
denotes the Gateaux derivatives of $A[\mu]$ as a mapping $D^{\star} \to \LC(D,B)$,
assuming that the definition of $A^{\al}[\mu]$ can be extended to a neighborhood of $M$ in $D^{\star}$.

This section is devoted to the preliminary analysis of the solutions to equation
\eqref{eqforderpar}. In the next section we shall explore their connections with the derivatives
 from the r.h.s. of \eqref{eqdefderwithpar}.

Let $\tilde D\subset D \subset B$ be, as above, three Banach spaces such that
  $\| \, \|_{\tilde D} \ge \| \, \|_D  \ge \| \, \|_B$,
 $D$ is dense in $B$ in the topology of $B$ and $\tilde D$ is dense in $D$ in the topology of $B$; and
let $M$ and $C_{\mu}([0,r],M)$ be defined as in Section \ref{secnonlinprop}.

\begin{theorem}
\label{thbasicnonlinearwellposevar}

 (i) Let, for each $\al$, $\xi \mapsto A^{\al}[\xi]$ be a
mapping from $B^{\star}$ to linear operators $A^{\al}[\xi]$ that
are uniformly bounded as operators $D\to B$
and $\tilde D\to D$ and
such that
\begin{equation}
\label{eqAisLiponxifam}
 \|A[\xi]-A[\eta]\|_{D\to B} \le
 c\|\xi-\eta\|_{D^{\star}}, \quad \xi, \eta \in
 B^{\star}
\end{equation}
for a constant $c>0$.

(ii)
For any $\al$, $\mu \in M$ and $\xi_.\in C_{\mu}([0,r],M)$,
 let the operator curve $A^{\al}[\xi_t]$ generate
  a strongly continuous backward propagator of uniformly bounded linear operators
 $U^{t,s;\al}[\xi_.]$, $0\le t\le s\le r$, in $B$ on the common invariant domain $D$,
 and with the dual propagator $V^{s,t;\al}[\xi_.]$ preserving the set $M$.

(iii) Let the propagators $\{U^{t,s;\al}[\xi_.]\}$, $t\le s$, are strongly
 continuous and bounded propagators in both $B$ and $D$.

(iv) Let the derivatives $\pa A^{\al}[\mu_t^{\al}]/ \pa \al$ exist in the norm topologies
of $\LC (D,B)$ and $\LC(\tilde D,D)$, and represent also bounded operators in
$\LC (D,B)$ and $\LC(\tilde D,D)$.

(v) Let $A^{\al}[\mu]$ can be extended to a
mapping $D^{\star} \to \LC(D,B)$ such that the limit in \eqref{eqdeffirstGatderwithinitial} exists
in the norm topology of $\LC(D,B)$ for any $\mu \in B^{\star}, \xi \in D^{\star}$. Moreover,
the  Gateaux derivatives $\xi \mapsto D_{\xi}A^{\al}[\mu]$ is continuous in $\mu$ (taken in the norm
topology of $B^{\star}$) and defines a bounded linear operator $D^{\star} \to \LC(D,B)$, that is
\begin{equation}
\label{eq1thbasicnonlinearwellposevar}
\|D_{\xi} A^{\al}[\mu]\|_{D\to B} \le c \|\mu\|_{B^{\star}}\|\xi \|_{D^{\star}}
\end{equation}
with a constant $c$.

 (vi) Finally, suppose there exists a representation
 \begin{equation}
\label{eqdualrepforfirstder}
(D_{\xi} A^{\al} [\mu]g,\mu)=(F^{\al}[\mu]g,\xi)
\end{equation}
with $F^{\al}[\mu]$ being a continuous mapping $D^{\star} \to \LC(D,D)$.

Then, for each $\al, \mu \in M$, there exists a weakly-$\star$ continuous in $D^{\star}$ family of propagator $\Pi^{s,t}[\al,\mu]$
 (constructed below) solving equation \eqref{eqforderpar} in $\tilde D^{\star}$, that is, for any
 $\xi_0 \in D^{\star}$, $\xi_t^{\al}=\Pi^{s,t}[\al,\mu]\xi_0$ satisfies \eqref{eqforderpar} for any $f\in \tilde D$.
\end{theorem}

\begin{remark}
Condition (vi) causes no trouble. In fact it follows from duality and additional weak continuity
assumption on $D_{\xi}$. We shall not formulate this assumption by two reasons.
(i) In case of reflexive $B$ it is satisfied automatically. (ii) Though in case we are most interested in, that is for $B^{\star}$ being the space of Borel measures, $B$ is not reflexive,
 in applications to Markov semigroup representation \eqref{eqdualrepforfirstder} again arises automatically,
due to the special structure of $A[\mu]$ (of the L\'evy-Khintchin type).
\end{remark}

\begin{remark}
Construction of propagators from condition (ii) can naturally be carried out via Theorem \ref{thbasicnonlinearwellpose1}, that is via $T$-products.
\end{remark}

\begin{proof}
Theorem \ref{thbasicnonlinearwellpose} implies that, for any $\al$,
 the weak nonlinear Cauchy problem
\begin{equation}
\label{eqgennonlinearprobparam}
 \frac{d}{dt}(f,\mu_t^{\al})=(A^{\al}[\mu_t^{\al}]f,\mu_t^{\al}), \quad \mu_0=\mu, \quad
 f\in D,
\end{equation}
is well posed in $M$, and its resolving semigroup $T_t^{\al}$ satisfies
\eqref{eqLipcontnonlinearsemigroup} uniformly in $\al$.

Next, the equation
\begin{equation}
\label{eqforderparred}
 \frac{d}{dt}(g,\xi_t(\al))
 =(A^{\al}[\mu_t^{\al}]g, \xi_t (\al))+(D_{\xi_t(\al)}A^{\al}[\mu_t^{\al}]g,\mu_t^{\al})
\end{equation}
has form \eqref{eqdefnonhomgeneratordual3} with $F_s$ specified by \eqref{eqdualrepforfirstder}, i.e.
\[
(F_sg,\xi)=(F^{\al}[\mu_s^{\al}]g,\xi)=(D_{\xi} A^{\al}[\mu_s^{\al}]g, \mu_s^{\al}).
\]

From \eqref{eq1thbasicnonlinearwellposevar} it follows that
\begin{equation}
\label{eq1thbasicnonlinearwellposevar1}
\|F_s\|_{D\to D}=\sup_{\|g\|_D\le 1} \sup_{\|\xi\|_{D^{\star}} \le 1}
(D_{\xi} A^{\al}[\mu_s^{\al}] g,\mu_s^{\al})
\le c \|\xi\|_{D^{\star}}\|\mu\|^2_{B^{\star}},
\end{equation}
which is uniformly bounded for $\mu_s^{\al} \in M$.
Consequently, Theorem \ref{thperturbweak} yields a construction of the strongly continuous family
$\{\Phi^{t,r}\}$ in $D$ such that its dual propagator $\{\Psi^{r,t}=(\Phi^{t,r})'\}$
solves the Cauchy problem for equation \eqref{eqforderparred}.

By the Duhamel principle, the solution to equation \eqref{eqforderpar}
for $r\ge t$ with the initial condition $\xi_t$
can be written as
\begin{equation}
\label{eqfinalpropforparamder}
(g,\Pi^{r,t}[\al,\mu]\xi_t)=(\Phi^{t,r}[\al,\mu]g, \xi_t) +\int_t^r
\left(\frac{\pa A^{\al}[\mu_s^{\al}]}{\pa \al} \Phi^{s,r}[\al,\mu]g,\mu_s^{\al}\right) \, ds.
\end{equation}
\end{proof}

\begin{theorem}
\label{thbasicnonlinearwellposevarstr}
Under the assumptions of Theorem \ref{thbasicnonlinearwellposevar},
assume additionally that the backward propagators $\{U^{t,s;\al}[\xi_.]\}$, $t\le s$, represent strongly
 continuous bounded propagators also in $\tilde D$
 (and hence, by the last statement of Theorem \ref{thdualpropagext1}, the family $A^{\al}[\xi_t]: D\to B$
  also generates  $\{U^{t,s;\al}\}$, as a propagator in $D$, on $\tilde D$).
  Then, for each $\al, \mu \in M, \xi_0 \in D^{\star}$, the curve $\Pi^{s,t}[\al,\mu]\xi_0$ represents the unique
  weakly-$\star$ continuous in $D^{\star}$ solution to equation \eqref{eqforderpar} in $\tilde D^{\star}$.
 \end{theorem}

 \begin{proof} This is a straightforward extension of Theorem \ref{thbasicnonlinearwellposevar}, obtained by taking into account the simple arguments given before Theorem \ref{thperturbweak}.
 \end{proof}

We shall not further pay attention to somewhat complicated details arising under the conditions
 of Theorem \ref{thbasicnonlinearwellposevar}, but will use more natural conditions of
Theorem \ref{thbasicnonlinearwellposevarstr}.

We complete this section by an additional stability result for $\Pi^{s,t}$.

\begin{theorem}
\label{thbasicnonlinearwellposevarcor}
Under the assumptions of Theorem \ref{thbasicnonlinearwellposevarstr},
suppose that

(i) in addition to \eqref{eqAisLiponxifam} and \eqref{eq1thbasicnonlinearwellposevar},
  one has the same properties for the pair $(\tilde D,D)$, i.e.
\begin{equation}
\label{eqAisLiponxifam1}
 \|A[\xi]-A[\eta]\|_{\tilde D\to D} \le
 c\|\xi-\eta\|_{D^{\star}}, \quad \xi, \eta \in
 B^{\star},
\end{equation}
\begin{equation}
\label{eq2thbasicnonlinearwellposevar}
\|D_{\xi} A^{\al}[\mu]\|_{\tilde D\to D} \le c \|\mu\|_{B^{\star}}\|\xi \|_{D^{\star}},
\end{equation}

(ii) derivatives of $A^{\al}[\mu]$ are Lipschitz in the norm-topology of $D^{\star}$,
more precisely:
\begin{equation}
\label{eq3thbasicnonlinearwellposevar}
\|\frac{\pa A^{\al}[\mu]}{\pa \al}-\frac{\pa A^{\al}[\eta]}{\pa \al}\|_{\tilde D\to D}
 \le c \|\mu-\eta\|_{D^{\star}},
\end{equation}
\begin{equation}
\label{eq4thbasicnonlinearwellposevar}
\|D_{\xi} (A^{\al}[\mu]-A^{\al}[\nu])\|_{D\to B} \le c \|\mu-\eta\|_{D^{\star}}\|\xi \|_{D^{\star}}.
\end{equation}

Suppose now that $\mu_0^{\al}(n) \to \mu_0^{\al}$ in the norm-topology of $D^{\star}$,
as $n\to \infty$ for each $\al$. Then
  $\Pi^{s,t}[\al,\mu_0^{\al}(n)]\xi_0 \to \Pi^{s,t}[\al,\mu_0^{\al}(n)]\xi_0$ weakly-$\star$
  in $D^{\star}$ and in the norm topology of $\tilde D^{\star}$.
\end{theorem}

\begin{proof}
We shall use the notation for propagators introduced above adding dependence on $n$ for all objects
constructed from $\mu_0^{\al}(n)$.

By \eqref{eqLipcontnonlinearsemigroup} we conclude that $T_t^{\al}\mu_0^{\al}(n) \to T_t^{\al}\mu_0^{\al}$,
as $n\to \infty$, in the norm-topology of $D^{\star}$ uniformly in $t,\al$. Hence, by \eqref{eqAisLiponxifam1}
and Theorem \ref{thconvergepropagators} (applied to the pair of spaces $(\tilde D,D)$),
\[
U^{t,s;\al}[T_.^{\al}\mu_0^{\al}(n)] \to U^{t,s;\al}[T_.^{\al}\mu_0^{\al}]
\]
in the norm-topology of $\LC (D,D)$. Similarly,
by \eqref{eq2thbasicnonlinearwellposevar} and \eqref{eq3thbasicnonlinearwellposevar},
\[
|(D_{\xi} A^{\al}[\mu]g,\mu)
-(D_{\xi} A^{\al}[\eta]g,\eta)|
\]
\[
\le |(D_{\xi} (A^{\al}[\mu]-A^{\al}[\eta])g,\mu)|+|(D_{\xi} A^{\al}[\eta]g,\mu-\eta)|
\]
\[
\le c\|\mu-\eta \|_{D^{\star}}\|g\|_{\tilde D} \|\xi\|_{D^{\star}}
 (\|\mu\|_{B^{\star}}+\|\eta\|_{B^{\star}}),
 \]
 so that
 \[
 \|F_s[\mu]g-F_s[\eta]g\|_D \le
c\|\mu-\eta \|_{D^{\star}}\|g\|_{\tilde D}
 (\|\mu\|_{B^{\star}}+\|\eta\|_{B^{\star}}).
 \]
and thus by Theorem \ref{thconvergepropagators},
\[
\Phi^{t,s}[\al, T_.^{\al}\mu_0^{\al}(n)] \to \Phi^{t,s}[\al, T_.^{\al}\mu_0^{\al}],
\quad n\to \infty,
\]
in the norm-topology of $\LC (D,D)$. Consequently, again by Theorem \ref{thconvergepropagators},
 \[
\Psi^{s,t}[\al, T_.^{\al}\mu_0^{\al}(n)]\xi \to \Psi^{s,t}[\al, T_.^{\al}\mu_0^{\al}]\xi
\]
weakly-$\star$ in $D^{\star}$ and in the norm-topology of $\tilde D^{\star}$, for any
$\xi \in D^{\star}$.

Finally, from \eqref{eqfinalpropforparamder} it follows that
\[
\left(g,\Pi^{r,0}[\al,\mu](n) \xi
-\Pi^{r,0}[\al,\mu] \xi\right)
\]
\[
=((\Phi^{0,r}(n)-\Phi^{0,r})g, \xi)
 + \int_0^r \left(\frac{\pa A^{\al}[\mu_s^{\al}(n)]}{\pa \al}
  (\Phi^{s,r}(n)-\Phi^{s,r})g, \mu_s^{\al}\right)
\]
\[
+\int_0^r \left(\frac{\pa A^{\al}[\mu_s^{\al}(n)]}{\pa \al}
  \Phi^{s,r}(n)g, \mu_s^{\al}(n)-\mu_s^{\al}\right)
  +\int_0^r \left(\left(\frac{\pa A^{\al}[\mu_s^{\al}(n)]}{\pa \al}
  -\frac{\pa A^{\al}[\mu_s^{\al}]}{\pa \al}\right)
  \Phi^{s,r}g, \mu_s^{\al}\right),
\]
which allows one to conclude that
\[
\|\Pi^{r,0}[\al,\mu](n) \xi
-\Pi^{r,0}[\al,\mu] \xi\|_{\tilde D^{\star}}
 \to 0,
\]
as $n\to \infty$, as required.
\end{proof}

\section{Sensitivity analysis for nonlinear propagators}

Our final question is whether the solution $\xi_t$ constructed in Theorem \ref{thbasicnonlinearwellposevar}
 does in fact yield the derivative \eqref{eqdefderwithpar}.
 The difference with the standard case, discussed in textbooks on ODE in Banach spaces, lies in the fact that
 the solution to the linearized equation \eqref{eqforderpar} exists in a different space that the nonlinear curve
 $\mu_t$ itself.

\begin{theorem}
\label{thnonlsensit}
Under the assumptions of Theorem \ref{thbasicnonlinearwellposevarcor},
let $\xi_0=\xi \in B^{\star}$ and is defined by \eqref{eqdefderwithparinit},
where the derivative exists in the norm-topology of $\tilde D^{\star}$ and weakly-${\star}$ in $D^{\star}$.
Then the unique solution $\xi_t[\al]= \Pi^{t,0}[\al,\mu^{\al}_0] \xi$
of equation \eqref{eqforderpar} constructed in the
Theorem \ref{thbasicnonlinearwellposevarstr} satisfies  \eqref{eqdefderwithpar},
where the derivative exists in the norm-topology of $\tilde D^{\star}$ and weakly-${\star}$ in $D^{\star}$.
\end{theorem}

\begin{proof}
The main idea is to approximate $A_s^{\al}$ by bounded operators, use the standard sensitivity theory for
vector valued ODE and then obtain the required result by passing to the limit.
To carry our this program,
let us pick up a family of operators $A_s^{\al}(n)$, $n=1,2,...$,
bounded in $B$ and $D$, that satisfy all the same conditions
as $A_s^{\al}$ and such that $\|(A_s^{\al}(n)-A_s^{\al})g\|_B \to 0$ for all $g\in D$ and uniformly for all $\al$ and
$g$  from bounded subsets of $\tilde D$. As such approximation, one can use either standard Iosida approximation
(which is convenient in abstract setting) or, in case of the generators of Feller Markov processes, generators
 of approximating pure-jump Markov processes.
As in the proof of Theorem \ref{thbasicnonlinearwellposevarcor},
we shall use the notation for propagators introduced in the previous section
 adding dependence on $n$ for all objects
constructed from $A_s^{\al}(n)$.

Since $(A_s^\al(n))'$ are bounded linear operators in
$B^{\star}$ and $D^{\star}$, the
equation for $\mu_t$ and $\xi_t$ are both well posed in the strong
sense in both $B^{\star}$ and $D^{\star}$.
 Hence the standard result on the differentiation with respect
 to initial data is applicable (see e.g. \cite{Mart} or Appendix D in \cite{Ko10nonlbook}) leading to the conclusion that
 $\xi_t[\al](n)$ represent the derivatives of $\mu_t^{\al}(n)$ in both $B^{\star}$ and $D^{\star}$.

  Consequently
 \begin{equation}
\label{eq1thnonlsensit}
 \mu_t^{\al}(n)-\mu_t^{\al_0}(n)= \int_{\al_0}^{\al} \xi_t[\be](n)\, d\beta
\end{equation}
holds as an equation in $D^{\star}$ (and in
$B^{\star}$ whenever $\xi\in B^{\star}$).

  Using Theorem
\ref{thbasicnonlinearwellposestabil} we
deduce the convergence of $\mu_t^{\al}(n)$ to $\mu^{\al}_t$ in the norm-topology of
$D^{\star}$. Consequently, using Theorem
\ref{thbasicnonlinearwellposevarcor} we can
deduce the convergence of $\xi_t^{\al}(n)$ to $\xi^{\al}_t$ in the norm-topology of
$\tilde D^{\star}$. Hence, we
can pass to the limit $n\to \infty$ in equation \eqref{eq1thnonlsensit} in the
norm topology of $\tilde D^{\star}$ yielding the equation
\begin{equation}
\label{eq2thnonlsensit}
 \mu_t^{\al}-\mu_t^{\al_0}= \int_{\al_0}^{\al} \xi_t[\be]\, d\beta,
\end{equation}
where all objects are well defined in
$(C^1_{\infty}(\R^d))^{\star}$.

This equation together with continuous dependence of
$\xi_t$ on $\al$ (which is proved in literally the same way as continuous dependence on $\mu$
 in Theorem \ref{thbasicnonlinearwellposevarcor}) implies \eqref{eqdefderwithpar} in the
sense required.
\end{proof}

Applying Theorem \ref{thnonlsensit} for the case of $A_s$ not depending on any additional parameter,
we obtain directly the smooth dependence of the nonlinear evolution $\mu_t$ on the initial data.
Namely, for $\mu_t=\mu_t(\mu_0)$, a solution to \eqref{eqkineqmeanfield1} with the initial condition $\mu_0$,
  we can define the Gateaux derivatives
\begin{equation}
\label{eqdeffirstGatderwithinitial}
 \xi_t(\mu_0,\xi) =D_{\xi}\mu_t(\mu_0)
 = \lim_{s \rightarrow 0_+} \frac{1}{s}(\mu_t (\mu_0 + s \xi) - \mu_t(\mu_0))
\end{equation}

 Differentiating \eqref{eqkineqmeanfield1}
with respect to initial data yields
\begin{equation}
\label{eqforfirstdergenkineqabs}
 \frac{d}{dt}(g,\xi_t(\mu_0,\xi))
 =(A[\mu_t]g, \xi_t (\mu_0,\xi))+(D_{\xi_t(\mu_0,\xi)}A[\mu_t]g,\mu_t),
\end{equation}
which represents a simple particular case of equation \eqref{eqforderpar}.
Hence, Theorem \ref{thnonlsensit} implies that, under the assumptions of this theorem (that do not involve
the dependence on $\al$), the derivative \eqref{eqdeffirstGatderwithinitial} does exists
and is given by the unique solution to equation \eqref{eqforfirstdergenkineqabs}
with the initial condition $\xi_0=\xi$, However, this existence and well-posedness hold
 weakly-$\star$ in $\tilde D^{\star}$, not in $B^{\star}$, as the nonlinear evolution itself.

 \section{Back to nonlinear Markov semigroups}

 We developed the theory in the most abstract form, for general nonlinear evolutions
 in Banach spaces, not even using positivity. This unified exposition allows one to obtain
 various concrete evolutions as a direct consequence of one general result. The main application we have in mind
 concerns the families $A[\mu]$ of the L\'evy-Kchintchin type form (with variable coefficients):
 \[
A[\mu]u(x)=\frac{1}{2} (G_{\mu}(x) \nabla, \nabla)u(x) +(b_{\mu}(x), \nabla u(x))
\]
\begin{equation}
\label{eqLevytypegenwithmes}
+\int[u(x+y)-u(x)-(y,\nabla u(x))\1_{B_1}(y)]\nu_{\mu} (x, dy),
\end{equation}
where $\nu_{\mu} (x,.)$ is a L\'evy measure for all $x\in \R^d, \mu \in \PC(\R^d)$.
The basic examples were given in the introduction.

Applied to {\it nonlinear L\'evy process} specified by the families
\eqref{eqLevygennonlinearrep}, our general results yield the following.

\begin{theorem}
\label{thexistencenonlinearLevy} Suppose the coefficients of a family
\eqref{eqLevygennonlinearrep} depend on $\mu$  Lipschitz continuously
in the norm of the Banach space $(C^2_{\infty}(\R^d))'$ dual to
$C^2_{\infty}(\R^d)$, i.e.
 \[
\|G(\mu)-G(\eta)\|+\|b(\mu)-b(\eta)\|
 +\int \min(1,|y|^2)|\nu(\mu,dy)-\nu(\eta,dy)|
 \]
\begin{equation}
\label{eqLipcondcoefnonlinearLevy}
  \le \kappa \|\mu-\eta\|_{(C^2_{\infty}(\R^d))'}
   =\kappa \sup_{\|f\|_{C^2_{\infty}(\R^d)}\le 1} |(f, \mu-\eta)|
\end{equation}
with constant $\kappa$. Then there exists a unique nonlinear
L\'evy semigroup generated by $A_{\mu}$, and hence a unique
nonlinear L\'evy process.
\end{theorem}

\begin{proof}
The well-posedness of all intermediate propagators is obvious in case of L\'evy processes, because they are constructed via Fourier transform, literally like L\'evy semigroup (details are given in \cite{Ko10nonlbook}).
Of course here $M=\PC(\R^d)$, $D=C_{\infty}^2(\R^d)$, $\tilde D=C_{\infty}^4(\R^d)$.
\end{proof}

\begin{remark} Condition \eqref{eqLipcondcoefnonlinearLevy} is not
at all weird. It is satisfied, for instance, when the coefficients
$G$,$b$, $\nu$ depend on $\mu$ via certain integrals (possibly
multiple) with smooth enough densities, i.e. in a way that is usually
met in applications.
\end{remark}

Applied to processes {\it of order at most one} specified by the families
\eqref{eqshortgeneratornonlin}, our general results yield the following.

\begin{theorem}
\label{thnonlinearshortgen}
 Assume that for any $\mu \in \PC
(\R^d)$, $b(.,\mu) \in C^1(\R^d)$ and $\nabla \nu (x,\mu, dy)$
(gradient with respect to $x$) exists in the weak sense as a signed
measure and depends weakly continuous on $x$. Let the following
conditions hold.

(i) boundedness:
 \begin{equation}
\label{eqshortcondlevynonlin3}
 \sup_{x,\mu} \int \min(1,|y|) \nu (x,\mu,dy)
 <\infty, \quad \sup_{x,\mu}  \int \min(1,|y|) |\nabla \nu (x,\mu,dy)|
 <\infty,
 \end{equation}

 (ii) tightness: for any
$\ep>0$ there exists a $K>0$ such that
\begin{equation}
\label{eqshortcondlevynonlin1}
 \sup_{x,\mu} \int_{\R^d\setminus B_K} \nu (x,\mu,dy)
 <\ep, \quad \sup_{x,\mu}  \int_{\R^d\setminus B_K} |\nabla \nu (x,\mu,dy)|
 <\ep,
 \end{equation}
 \begin{equation}
\label{eqshortcondlevynonlin2}
 \sup_{x,\mu} \int_{B_{1/K}} |y|\nu (x,\mu,dy)< \ep,
 \end{equation}

 (iii) Lipschitz continuity:
\begin{equation}
\label{eqshortcondlevynonlin4}
 \sup_{x} \int \min(1,|y|) |\nu (x,\mu_1,dy)-\nu(x,\mu_2,dy)|
 \le c\|\mu_1-\mu_2\|_{(C^1_{\infty}(\R^d))^{\star}},
 \end{equation}

\begin{equation}
\label{eqshortcondlevynonlin5}
 \sup_{x} |b(x,\mu_1)-b(x,\mu_2)|
 \le c\|\mu_1-\mu_2\|_{(C^1_{\infty}(\R^d))^{\star}}
 \end{equation}
uniformly for bounded $\mu_1, \mu_2$.

Then the weak nonlinear Cauchy problem
\eqref{eqkineqmeanfield1} with $A_{\mu}$ given by
\eqref{eqshortgeneratornonlin} is well
posed, i.e. for any $\mu \in \MC(\R^d)$ it has a
unique solution $T_t(\mu) \in \MC (\R^d)$ (so that
\eqref{eqshortgeneratornonlin} holds for all $g\in
C^1_{\infty}(\R^d)$) preserving the norm, and the transformations
$T_t$ of $\PC (\R^d)$, $t\ge 0$, form
a semigroup depending Lipschitz continuously on time $t$ and the
initial data in the norm of $(C^1_{\infty}(\R^d))^{\star}$.
\end{theorem}

\begin{proof} Here we use
$M=\PC(\R^d)$, $D=C_{\infty}^1(\R^d)$, $\tilde D=C_{\infty}^2(\R^d)$.
The corresponding auxiliary propagators required in Theorem \ref{thdualpropag}
are constructed in \cite{Ko10nonlbook} (Chapter 4) and \cite{Ko11Markbook} (Chapter 5).
\end{proof}

In both cases above, straightforward additional smoothness assumptions on the coefficients
of the generator yield smoothness with respect to parameters and/or initial data
via Theorem \ref{thnonlsensit}.

Similarly one gets the well-posedness for
mixtures of nonlinear diffusions and stable-like processes given
by \eqref{eqgendifstablelikedegenper} with coefficients depending on distribution $\mu$.
Our theory also applies to nonlinear stable-like processes on manifolds, see \cite{Ko10nonlbook}
(Section 11.4), and to nonlinear dynamic quantum semigroups, see \cite{Ko10nonlbook}
(Section 11.3).

Let us stress again, referring to
\cite{Ko10nonlbook}, \cite{Ko08} and \cite{Ko07}, that the first and second derivatives of nonlinear Markov semigroups
with respect to initial data (for simplicity, we dealt only with the first derivative here) describe
the dynamic law of large numbers for interacting particle systems and the corresponding central limit theorem for fluctuations, respectively.

\end{document}